\documentclass[letterpaper, 10 pt, conference]{ieeeconf}  

\IEEEoverridecommandlockouts                              
\usepackage{graphics} 
\usepackage{epsfig} 
\usepackage{mathptmx} 
\usepackage{times} 
\usepackage{amsmath} 
\usepackage{amssymb}  
\usepackage{xcolor}
\usepackage{comment}
\usepackage{empheq}

\usepackage{amsthm}
\usepackage{microtype}

\usepackage{tikz}
\usepackage{pgfplots}
\pgfplotsset{compat=1.18}
\usetikzlibrary{backgrounds}
\usetikzlibrary{intersections}
\usepgfplotslibrary{fillbetween}
\usepackage{import}

\usepackage{printlen}
\uselengthunit{mm}

\usepackage[%
  backend=biber,
  url=true,
  style=ieee, 
  sorting=none, 
  maxnames=4,
  minnames=3,
  maxbibnames=99,
  giveninits,
  uniquename=init]{biblatex} 

\newtheorem{definition}{Definition}
\newtheorem{theorem}{Theorem}

\title{\LARGE \bf
Implicit Primal-Dual Interior-Point Methods\\for Quadratic Programming}

\author{Jon Arrizabalaga and Zachary Manchester
\thanks{Jon Arrizabalaga and Zachary Manchester are with the Department of Aeronautics and Astronautics, Massachusetts Institute of Technology, USA
        {\tt\small jonarri@mit.edu, zacm@mit.edu }}%
}

\newcommand{\inred}[1]{{\color{red}#1}}
\newcommand{\normal}{} 

\bibliography{references}

\begin{document}

\maketitle
\thispagestyle{empty}
\pagestyle{empty}

\begin{abstract}
This paper introduces a new method for solving quadratic programs using primal-dual interior-point methods. Instead of handling complementarity as an explicit equation in the Karush-Kuhn-Tucker (KKT) conditions, we ensure that complementarity is implicitly satisfied by construction. This is achieved by introducing an auxiliary variable and relating it to the duals and slacks via a retraction map. Specifically, we prove that the softplus function has favorable numerical properties compared to the commonly used exponential map. The resulting KKT system is guaranteed to be spectrally bounded, thereby eliminating the most pressing limitation of primal-dual methods: ill-conditioning near the solution. These attributes facilitate the solution of the underlying linear system, either by removing the need to compute factorizations at every iteration, enabling factorization-free approaches like indirect solvers, or allowing the solver to achieve high accuracy in low-precision arithmetic. Consequently, this novel perspective opens new opportunities for interior-point methods, especially for solving large-scale problems to high precision.
\end{abstract}

\section{Introduction}
Primal-dual interior-point methods (PDIPMs) have long served as the workhorse of convex optimization, prized for their rapid convergence and robust theoretical guarantees~\cite{wright1997primal}. However, as iterates approach the optimal solution, the underlying Karush-Kuhn-Tucker (KKT) system inevitably becomes ill-conditioned~\cite{wright1998ill}. This forces expensive, exact matrix factorizations at every iteration, rendering traditional PDIPMs prohibitively costly for large-scale problems~\cite{nocedal2006numerical}.

To overcome this computational bottleneck, researchers have broadly pursued two distinct trajectories. The first accelerates PDIPMs using inexact interior-point methods, replacing exact matrix factorizations with iterative linear solvers. While this reduces per-iteration costs, iterative solvers are highly sensitive to the ill-conditioning of standard KKT systems, necessitating complex, expensive preconditioners to prevent stagnation and numerical instability during final iterations.

Alternatively, the second trajectory abandons the interior-point framework entirely in favor of first-order methods, most notably the Alternating Direction Method of Multipliers (ADMM)~\cite{neal2011distributed} and Primal-Dual Hybrid Gradient (PDHG) algorithms~\cite{chambolle2011first}. These approaches are highly scalable and boast remarkably cheap iterations. However, they suffer from notoriously poor tail convergence, and thus, achieving the high precision characteristic of PDIPMs requires an impractical number of iterations.

Given this dichotomy—where interior-point methods converge to high accuracy in a few expensive iterations, and first-order methods require many cheap iterations but yield low accuracy—a very reasonable question arises: \emph{Is it possible to develop a method that achieves the low iteration cost of first-order approaches while maintaining the rapid convergence and high precision of interior-point methods?}

\begin{figure}[t]
	\centering
    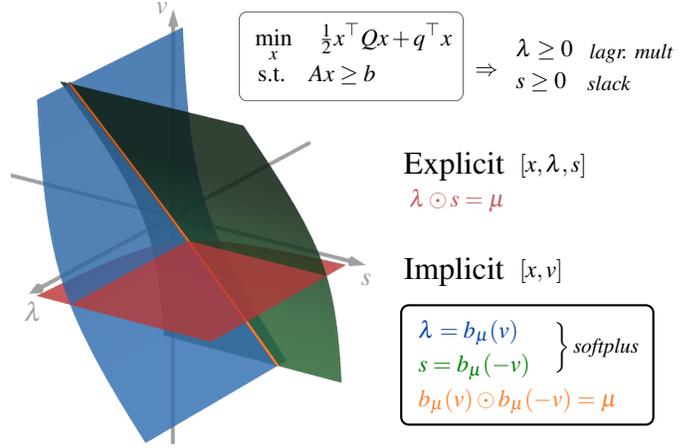
	\caption{We present an implicit method for primal-dual interior-point optimization. Unlike the standard approach that enforces complementarity as an explicit algebraic constraint (red), our method satisfies it by construction (orange) using an auxiliary variable $v$ and a softplus retraction map $b_\mu(v)$ (green, blue). This structural shift guarantees a spectrally bounded KKT system, even in the vicinity of a solution. By eliminating the eigenvalue divergence of traditional methods, our approach enables reusing matrix factorizations, adopting factorization-free indirect solvers, and reaching high accuracy with low-precision arithmetic.}\label{fig:explanation}
    \vspace{-6mm}
\end{figure}

In this manuscript, we address this question by presenting an implicit method for solving primal-dual interior-point-based quadratic programs (Fig.~\ref{fig:explanation}). In particular, our primary contributions are:
\begin{enumerate}
    \item The introduction of an auxiliary variable and a softplus retraction map that reformulates the KKT conditions into a spectrally bounded linear system, preventing the eigenvalue divergence characteristic of standard interior-point methods.
    \item A theoretical proof demonstrating why the softplus retraction map is the unique map capable of guaranteeing the spectral boundedness of the underlying KKT system.
    \item Preliminary results showcasing how our implicit method succeeds where standard interior-point methods have traditionally struggled, specifically enabling the reuse of matrix factorizations for inexact Newton steps, facilitating factorization-free iterative solvers, and achieving high accuracy in low-precision arithmetic.
\end{enumerate}
\newpage
The remainder of this paper is organized as follows: Section~\ref{sec:preliminaries} reviews the fundamentals of quadratic programming and the standard primal-dual interior-point framework. Section~\ref{sec:implicit} introduces our proposed implicit complementarity formulation and discusses its theoretical properties and conceptual advantages. These benefits are numerically validated through the experiments presented in Section~\ref{sec:experiments}. The method's broader impact and future research directions are discussed in Section~\ref{sec:future}, followed by concluding remarks in Section~\ref{sec:conclusion}.

\section{Preliminaries}\label{sec:preliminaries}
In this section, we formally introduce quadratic programming and the standard (explicit) primal-dual interior-point framework. This will set the stage for the next section, where we introduce our proposed implicit formulation.
\subsection{Quadratic Programming}
A standard convex Quadratic Program (QP) is given by the following primal-dual pair:
\begin{subequations} \label{eq:qp_problems_eq}
\begin{alignat}{2}
    &\text{Primal:} \quad && \begin{aligned}[t]
        & \min_{x} && \tfrac{1}{2} x^\top Q x + q^\top x \\
        & \text{s.t.} && Ax \ge b \\
        & && Cx = d
    \end{aligned} \label{eq:qp_primal_eq} \\[1em]
    &\text{Dual:} \quad && \begin{aligned}[t]
        & \max_{x,\lambda,\gamma} && -\tfrac{1}{2} x^\top Q x + b^\top \lambda + d^\top \gamma \\
        & \text{s.t.} && Qx + q - A^\top \lambda - C^\top \gamma = 0 \\
        & && \lambda \ge 0
    \end{aligned} \label{eq:qp_dual_eq}
\end{alignat}
\end{subequations}
where $x\in\mathbb{R}^n$ is the primal variable, $s\in\mathbb{R}^m$ is the slack variable, $\lambda\in\mathbb{R}^m$ is the inequality-dual variable, and $\gamma\in\mathbb{R}^p$ is the equality-dual variable. The problem data consists of the symmetric positive semi-definite objective matrix $Q \in \mathbb{R}^{n \times n}$ ($Q \succeq 0$), the linear objective vector $q \in \mathbb{R}^n$, the inequality constraint matrix $A \in \mathbb{R}^{m \times n}$ with bounds $b \in \mathbb{R}^m$, and the equality constraint matrix $C \in \mathbb{R}^{p \times n}$ with bounds $d \in \mathbb{R}^p$. The Karush-Kuhn-Tucker (KKT) conditions for optimality are: 
\begin{subequations} \label{eq:kkt_conditions}
\begin{align}
    &\text{Stationarity: } & Qx + q - A^\top \lambda - C^\top \gamma = 0 \\
    &\text{Primal Feasibility: } & \begin{aligned}[t] 
                                    Ax - b - s &= 0, \quad s \ge 0 \\ 
                                    Cx - d &= 0 
                                   \end{aligned} \\
    &\text{Dual Feasibility: } & \lambda \ge 0 \\
    &\text{Complementarity: } & \lambda \odot s = 0 \label{eq:kkt_complementarity}
\end{align}
\end{subequations}
where $\odot$ denotes the element-wise product. Solving the QP is equivalent to finding the root of this system of equations subject to the non-negativity constraints on $\lambda$ and $s$. To find the primal, dual and slack variables that satisfy eqs.~\eqref{eq:kkt_conditions}, we define the global state vector 
\begin{equation*}
z = \begin{bmatrix}x^\top& \lambda^\top&  \gamma^\top& s^\top\end{bmatrix}^\top
\end{equation*}
and the residual vector as
\begin{equation}\label{eq:residuals}
    r(z) = 
    \begin{bmatrix}
        r_x \\
        r_i \\
        r_e\\
        r_c
    \end{bmatrix}
    =
    \begin{bmatrix}
        Qx + q - A^\top \lambda - C^\top \gamma \\
        Ax - b - s \\
        Cx-d\\
        \lambda\odot s
    \end{bmatrix}\,.
\end{equation}
\subsection{Primal-Dual Interior-Point Methods}
Primal-dual interior-point methods relax the complementarity condition as
\begin{equation}\label{eq:complementarity}
    r_{c,\mu}(z) = \lambda\odot s - \mu
\end{equation}
and thus, the smoothed version of the residual in~\eqref{eq:residuals} becomes
\begin{equation}\label{eq:pd_residuals}
    r_\mu(z) = \begin{bmatrix} r_x^\top& r_i^\top & r_e^\top & r_{c,\mu}^\top\end{bmatrix}^\top\,.
\end{equation}
Generally, primal-dual interior-point methods operate by taking Newton steps that simultaneously aim to drive the smoothed residual $r_\mu(z)$ to zero while systematically reducing the barrier parameter $\mu$, leading to $r(z^*) = \left. r_\mu(z^*) \right|_{\mu=0} \approx 0$. In other words, the algorithm converges when the total duality gap,
\begin{equation*}
    \eta = \lambda^\top s \,,
\end{equation*}
falls below a specified tolerance. In practice, $\mu$ is dynamically updated as a fraction of the average gap:
\begin{equation}\label{eq:mu_reduction}
    \mu = \sigma \frac{\eta}{m} \,,
\end{equation}
where $\sigma \in (0, 1]$ is a centering parameter controlling the targeted gap reduction. For a given $\mu$ the solution to~\eqref{eq:pd_residuals} is obtained through Newton's method:
\begin{equation*}
    r_\mu(z+\Delta z) \approx r_\mu(z) + \nabla_zr_\mu(z) \Delta z = 0\,.
\end{equation*}
Given the residual definition in~\eqref{eq:pd_residuals}, this leads to the following linear system,
\begin{equation*} 
\begin{bmatrix}
Q & -A^\top & -C^\top & 0 \\
A & 0 & 0 & -I \\
C & 0 & 0 & 0 \\
0 & S & 0 & \Lambda \\
\end{bmatrix}
\begin{bmatrix}
\Delta x \\
\Delta \lambda \\
\Delta \gamma \\
\Delta s
\end{bmatrix}
= -
\begin{bmatrix}
r_x \\
r_i \\
r_e \\
r_{c,\mu}
\end{bmatrix} 
\, ,
\end{equation*}
with $S=\text{diag}(s)$ and $\Lambda=\text{diag}(\lambda)$. After some simplifications, it can be reduced into the following form:
\begin{equation} \label{eq:explicit_partially_condensed}
\underbrace{
\begin{bmatrix}
Q & -A^\top & -C^\top \\
-A & -\Lambda^{-1} S & 0 \\
-C & 0 & 0
\end{bmatrix}
}_{E(\lambda, s)}
\begin{bmatrix}
\Delta x \\
\Delta \lambda \\
\Delta \gamma
\end{bmatrix}
=
\begin{bmatrix}
-r_x \\
r_i + \Lambda^{-1} r_{c,\mu} \\
r_{e}
\end{bmatrix} \, ,
\end{equation}
Throughout this paper, we will refer to this standard approach as the \textit{explicit} primal-dual method, because the complementarity condition appears as an explicit algebraic expression.

The fundamental limitation of this traditional formulation stems from the behavior of the matrix $E(\lambda, s)$ near the optimal solution. As strict complementarity dictates that $\lambda \odot s \rightarrow 0$, the elements of the diagonal matrix $\Lambda^{-1}S$ experience extreme scaling disparities---either blowing up to infinity (when $\lambda_i \rightarrow 0$) or vanishing to zero (when $s_i \rightarrow 0$). While the resulting ill-conditioning is widely known to be benign \cite{benzi2005numerical} and can be circumvented via exact matrix factorizations, the per-iteration cost of these factorizations remains a significant computational bottleneck. Attempts to avoid this cost typically rely on exploiting tight eigenvalue clustering, a property that deteriorates as the matrix eigenvalues diverge~\cite{gondzio2012matrix}.

The following section introduces our \emph{implicit} approach to the complementarity condition, which ---among other advantages--- ensures a bounded eigenvalue spectrum.


\section{Implicit Complementarity}~\label{sec:implicit}
To satisfy the complementarity condition in~\eqref{eq:complementarity}, we adopt the alternative approach pioneered by~\cite{permenter2023log, permenter2023geodesic}. Instead of treating complementarity as an explicit equation that we solve through Newton's method, we augment our system with an auxiliary variable $v \in \mathbb{R}^m$. Accordingly, our new state vector becomes:
\begin{equation*}
    z = \begin{bmatrix} x^\top & \lambda^\top & \gamma^\top & s^\top & v^\top \end{bmatrix}^\top\,.
\end{equation*}
This auxiliary variable is designed to enforce complementarity implicitly, by construction. To achieve this, we introduce a retraction map:
\begin{equation}\label{eq:retraction_map}
    b_\mu : \mathbb{R} \rightarrow \mathbb{R}_{+}\;\;\text{such that}\;\;b_\mu(v) \cdot b_\mu(-v) = \mu\,.
\end{equation}
While the works of~\cite{permenter2023log, permenter2023geodesic} explicitly select the exponential function $b_\mu(v)=\sqrt{\mu}e^v$ for the mapping in~\eqref{eq:retraction_map}, we make no such initial assumptions. Instead, we maintain $b_\mu$ as a generic function, whose optimal form will be formally established later. By enforcing $\lambda = b_\mu(v)$ and $s = b_\mu(-v)$, we replace the explicit complementarity residual $r_{c, \mu}$ with two new residuals:
\begin{equation}\label{eq:implicit_res_components}
    r_{\lambda,\mu} = \lambda - b_\mu(v) \, , \quad r_{s,\mu} = s - b_\mu(-v) \, .
\end{equation}
Consequently, the complete residual vector for our implicit formulation is defined as:
\begin{equation*} \label{eq:impl_residuals}
    r_\mu(z) = \begin{bmatrix} r_x^\top & r_i^\top & r_e^\top & r_{\lambda, \mu}^\top & r_{s, \mu}^\top \end{bmatrix}^\top \, .
\end{equation*}
To formulate the Newton step, we must linearize the implicit complementarity residuals defined in~\eqref{eq:implicit_res_components}. Applying the chain rule we obtain:
\begin{subequations}
\begin{align*}
    \nabla_v r_{\lambda,\mu} &= -\nabla_v (b_\mu(v)) = -db_\mu(v)\,, \\
    \nabla_v r_{s,\mu} &= -\nabla_v (b_\mu(-v)) = db_\mu(-v)\,,
\end{align*}
\end{subequations}
where $db_\mu(\cdot)$ denotes the derivative of $b_\mu$ with respect to its argument. To express these gradients in matrix form, we define the following diagonal matrices:
\begin{equation*}
    B_\mu(v) = \text{diag}(db_\mu(v))\,.
\end{equation*}
The resulting KKT system is
\begin{equation*}\label{eq:implicit_uncondensed}
\begin{bmatrix}
Q & -A^\top & -C^\top & 0 & 0 \\
A & 0 & 0 & -I & 0 \\
C & 0 & 0 & 0 & 0 \\
0 & I & 0 & 0 & -B_{\mu}(v) \\
0 & 0 & 0 & I & B_{\mu}(-v) \\
\end{bmatrix}
\begin{bmatrix}
\Delta x \\
\Delta \lambda \\
\Delta \gamma \\
\Delta s \\
\Delta v
\end{bmatrix} \\
= -
\begin{bmatrix}
r_x \\
r_i \\
r_e \\
r_{\lambda,\mu}\\
r_{s,\mu}
\end{bmatrix}\,,
\end{equation*}
which can be reduced to
\begin{multline}\label{eq:implicit_partially_condensed}
    \underbrace{
    \begin{bmatrix}
        Q - A^\top A & -A^\top \left(B_{\mu}(v)+B_{\mu}(-v)\right) & -C^\top \\
        -A & -B_{\mu}(-v) & 0 \\
        -C & 0 & 0 \\
    \end{bmatrix}
    }_{J(v)}
    \begin{bmatrix}
        \Delta x\\
        \Delta v\\
        \Delta \gamma
    \end{bmatrix}
    =\\                                       
    \begin{bmatrix}
        -r_{x}+A^\top (r_{i}-r_{\lambda,\mu}+r_{s,\mu})\\
        r_{i} + r_{s,\mu}\\
        r_{e}
    \end{bmatrix}\,.
\end{multline}
We denote  $J(v)$ as the \emph{implicit matrix}, which depends on variable $v$ and the---yet to be defined---retraction map $b_\mu$ in eq.~\eqref{eq:retraction_map}. 


\begin{figure}[t]
	\centering
    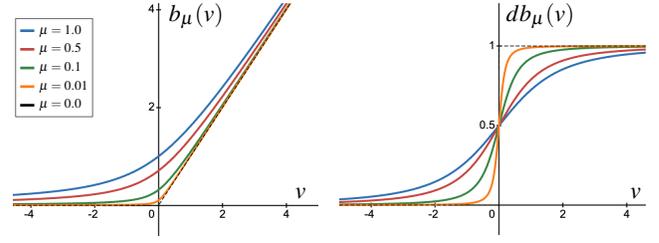
	\caption{The softplus retraction map (left) and its derivative (right) for varying values of the barrier parameter $\mu$. As proven in Theorem~\ref{thm:unique_retraction_map}, the softplus is specifically chosen because it is the unique retraction map that fulfills the criteria outlined in Definition~\ref{def:retraction_map}. Notably, its derivative is strictly bounded between 0 and 1 (right). Such boundedness guarantees that the eigenvalue spectrum of the linear system stays within a compact set as the solution is approached. See Fig.~\ref{fig:explanation} for a spatial illustration of how two softplus retraction maps implicitly enforce complementarity.}\label{fig:softplus}
    \vspace{-4mm}
\end{figure}


\subsection{Choosing the retraction map}
We now turn to the design of the retraction map $b_\mu$. Aiming to find the most appropriate mapping, rather than selecting an arbitrary function or defaulting to the exponential map of~\cite{permenter2023log, permenter2023geodesic}, we take a principled approach. We observe the implicit matrix in~\eqref{eq:implicit_partially_condensed} and identify two desirable properties the retraction map should hold to facilitate solving the respective linear system. Once defined, we proceed to find the map that satisfies them.

First, observe that if the sum of the derivative matrices simplifies to the identity—that is, 
\begin{equation}\label{eq:property1_matrix}
    B_{\mu}(v) + B_{\mu}(-v) = I\,,
\end{equation}
then the implicit matrix $J(v)$ cleanly reduces to:
\begin{equation}\label{eq:implicit_partially_condensed_softrelu}
J(v)= 
\begin{bmatrix}
Q - A^\top A & -A^\top & -C^\top \\
-A & -B_{\mu}(-v) & 0 \\
-C & 0 & 0 \\
\end{bmatrix}\,,
\end{equation}
which restores the structural symmetry.
Second, if we constrain the retraction map such that its derivative $B_{\mu}(-v)$ remains bounded, we secure two critical numerical advantages: (i) the corresponding eigenvalues of the implicit matrix $J(v)$ remain strictly bounded by the derivative range of the retraction map, a property that the exponential map used in~\cite{permenter2023log, permenter2023geodesic} lacks; and (ii) by maintaining bounded matrix entries, this formulation avoids the inherent eigenvalue divergence to infinity that characterizes the explicit matrix $E(\lambda, s)$ in~\eqref{eq:explicit_partially_condensed} as iterates approach the optimal solution.



Building upon the initial property in~\eqref{eq:retraction_map}, and guided by these new structural and numerical requirements, we now formalize the complete definition of the retraction map:

\begin{definition}[Retraction map] 
\label{def:retraction_map}
Let $b_\mu : \mathbb{R} \rightarrow \mathbb{R}_{+}$ be a map that satisfies the following properties
\begin{subequations}
\begin{align}
    &b_{\mu}(v) \cdot b_{\mu}(-v) = \mu\,, \label{def:retraction_map_1}\\
    &db_{\mu}(v) + db_{\mu}(-v) = 1\,, \label{def:retraction_map_2}\\
    &0 < db_{\mu}(v) \leq 1 \label{def:retraction_map_3}
\end{align}
\end{subequations}
then, $b_\mu$ is denoted as a retraction map.
\end{definition}
Notice that the exponential map in~\cite{permenter2023log, permenter2023geodesic} only satisfies the original complementarity requirement~\eqref{def:retraction_map_1} and does not comply with~\eqref{def:retraction_map_2} and \eqref{def:retraction_map_3}.

Having established the properties required for the retraction map to guarantee a spectrally bounded system, a natural question arises: \emph{does such a function exist, and is it unique?} The following theorem answers this affirmatively, proving that the softplus is the sole function fulfilling this definition.
\begin{theorem} \label{thm:unique_retraction_map}
Let $b_\mu : \mathbb{R} \rightarrow \mathbb{R}_+$ be a retraction map satisfying the properties in Definition~\ref{def:retraction_map}. Then, for any $\mu > 0$, the unique function that fulfills these conditions is the softplus function, given by
\begin{equation}\label{eq:softplus}
    b_\mu(v) = \frac{v + \sqrt{v^2 + 4\mu}}{2} \, .
\end{equation}
\end{theorem}
A visual representation of this function and its strictly bounded derivative is provided in Fig.~\ref{fig:softplus}.

\begin{proof}
From~\eqref{def:retraction_map_2}, we have $db_\mu(v) + db_\mu(-v) = 1$. Integrating this equation with respect to $v$ yields
\begin{equation}
    b_\mu(v) - b_\mu(-v) = v + C \, , \label{eq:proof_integration}
\end{equation}
where $C$ is an integration constant. Evaluating~\eqref{eq:proof_integration} at $v=0$ yields $b_\mu(0) - b_\mu(0) = C$, implying $C = 0$. Thus, $b_\mu(-v) = b_\mu(v) - v$. Substituting this result into~\eqref{def:retraction_map_1}, $b_\mu(v) \cdot b_\mu(-v) = \mu$, we obtain a quadratic equation in terms of $b_\mu(v)$:
\begin{equation*}
    b_\mu(v)^2 - v b_\mu(v) - \mu = 0 \, .
\end{equation*}
Solving for $b_\mu(v)$ using the quadratic formula yields
\begin{equation*}
    b_\mu(v) = \frac{v \pm \sqrt{v^2 + 4\mu}}{2} \, .
\end{equation*}
By Definition~\ref{def:retraction_map}, the retraction map maps to strictly positive real numbers ($b_\mu : \mathbb{R} \to \mathbb{R}_+$). Since $\sqrt{v^2 + 4\mu} > |v|$ for any $\mu > 0$, the negative root yields $b_\mu(v) < 0$ for all $v \in \mathbb{R}$. Thus, we must select the positive root. Finally, we verify that this unique solution satisfies~\eqref{def:retraction_map_3}. The derivative is
\begin{equation*}
    db_\mu(v) = \frac{1}{2} \left( 1 + \frac{v}{\sqrt{v^2 + 4\mu}} \right) \, .
\end{equation*}
Because the fractional term is strictly bounded as $\frac{v}{\sqrt{v^2 + 4\mu}} \in (-1, 1)$ for all $v \in \mathbb{R}$ and $\mu > 0$, it follows directly that $0 < db_\mu(v) < 1$. This satisfies the bounding condition and concludes the proof. 
\end{proof}

\subsection{Superiority of Softplus over Exponential}
Existing implicit interior-point methods~\cite{permenter2023geodesic,permenter2023log} typically rely on the exponential retraction map, $b_\mu(v) = \sqrt{\mu}\exp(v)$. While this satisfies the basic complementarity condition, the proposed softplus map in eq.~\eqref{eq:softplus} offers two fundamental numerical advantages.

First, as established in Definition 1, the softplus map possesses a strictly bounded derivative that symmetrically sums to the identity. These properties guarantee that the underlying KKT matrix in~\eqref{eq:implicit_partially_condensed_softrelu} remains symmetric and spectrally bounded. The exponential map lacks these properties, allowing eigenvalues to diverge near the solution.

Second, the softplus map prevents numerical shocks during barrier updates ($\mu^+ = \sigma\mu$, with $\sigma \in (0,1)$). In the exponential map, $\mu$ acts as a global multiplier: reducing the barrier artificially shrinks all constraint mappings by $\sqrt{\sigma}$. The softplus map elegantly resolves this by natively decoupling the barrier parameter from active constraints while correctly scaling inactive ones. We formalize this in the following theorem.

\begin{theorem}\label{thm:barrier_reduction}
Let $\mu^+ = \sigma\mu$ with $\sigma \in (0,1)$, and define $\lambda_\mu(v) = b_\mu(v)$ and $s_\mu(v) = b_\mu(-v)$, where $b_\mu$ is the softplus map in (15). For any fixed $v \neq 0$:
\begin{itemize}
    \item If $v > 0$: $0 < s_{\mu^+}(v) \le \frac{\sigma\mu}{v}$ and $0 \le \lambda_\mu(v) - \lambda_{\mu^+}(v) \le \frac{(1-\sigma)\mu}{v}$.
    \item If $v < 0$: $0 < \lambda_{\mu^+}(v) \le \frac{\sigma\mu}{|v|}$ and $0 \le s_\mu(v) - s_{\mu^+}(v) \le \frac{(1-\sigma)\mu}{|v|}$.
\end{itemize}
\end{theorem}
Consequently, once the sign of $v$ identifies the active side, reducing the barrier leaves the dominant variable nearly unchanged and only modifies the complementary variable by $\mathcal{O}(\mu)$. Conversely, if we were to instead use the exponential map $b_\mu(v) = \sqrt{\mu}\exp(v)$, reducing the barrier uniformly rescales both variables: $\lambda_{\mu^+}(v) = \sqrt{\sigma}\lambda_\mu(v)$ and $s_{\mu^+}(v) = \sqrt{\sigma}s_\mu(v)$

\begin{proof}
From eq.~\eqref{eq:softplus}, the softplus map satisfies $b_\mu(v) - b_\mu(-v) = v$ and $b_\mu(v)b_\mu(-v) = \mu$. 

Assume first that $v > 0$. By definition, we can write $\lambda_{\mu^+}(v) = v + s_{\mu^+}(v)$ and $s_{\mu^+}(v) = \mu^+ / \lambda_{\mu^+}(v)$. Because $s_{\mu^+}(v) > 0$, it is clear that $\lambda_{\mu^+}(v) \ge v$. Substituting this into the expression for $s_{\mu^+}(v)$ yields the first bound:
\begin{equation*}
    0 < s_{\mu^+}(v) = \frac{\sigma\mu}{\lambda_{\mu^+}(v)} \le \frac{\sigma\mu}{v}.
\end{equation*}

Next, we evaluate the difference $\lambda_\mu(v) - \lambda_{\mu^+}(v)$. From (15), $b_\mu(v)$ is monotonically increasing in $\mu$, which implies $\lambda_{\mu^+}(v) \le \lambda_\mu(v)$. Thereby, we can bound $s_{\mu^+}(v)$ as:
\begin{equation*}
    s_{\mu^+}(v) = \frac{\sigma\mu}{\lambda_{\mu^+}(v)} \ge \frac{\sigma\mu}{\lambda_\mu(v)} = \sigma s_\mu(v).
\end{equation*}

Since the difference in $\lambda$ equals the difference in $s$ (because $\lambda - s = v$ is constant for a fixed $v$), we have:
\begin{align*}
    0 \le \lambda_\mu(v) - \lambda_{\mu^+}(v) &= s_\mu(v) - s_{\mu^+}(v) \\
    &\le s_\mu(v) - \sigma s_\mu(v) = (1-\sigma)s_\mu(v).
\end{align*}

Finally, because $v > 0$ implies $\lambda_\mu(v) \ge v$, we know $s_\mu(v) = \mu / \lambda_\mu(v) \le \mu / v$. Substituting this into the inequality above provides the second bound:
\begin{equation*}
    0 \le \lambda_\mu(v) - \lambda_{\mu^+}(v) \le \frac{(1-\sigma)\mu}{v}.
\end{equation*}

The case $v < 0$ follows by applying the exact same argument to $-v$, exchanging the roles of $\lambda$ and $s$. 

For the exponential map $b_\mu(v)=\sqrt{\mu}e^v$, the uniform scaling property follows immediately by direct substitution:
\begin{equation*}
    \lambda_{\mu^+}(v) = \sqrt{\sigma}\lambda_\mu(v), \quad
    s_{\mu^+}(v) = \sqrt{\sigma}s_\mu(v).
\end{equation*}
This concludes the proof.
\end{proof}
After theoretically motivating the appropriateness of the softplus as the retraction map for handling complementarity implicitly, in the next section we proceed to validate it numerically.

\begin{figure*}[t]
    \centering
    \input{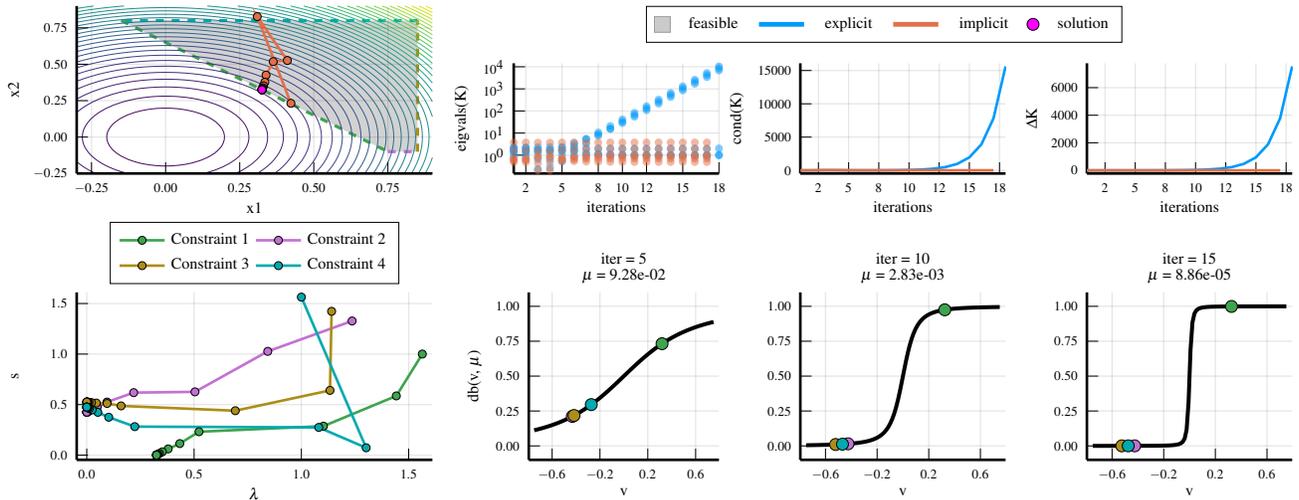}
    \vspace{-5mm}
    \caption{Fundamental numerical properties of the implicit representation evaluated on the synthetic toy problem. \textbf{Column 1:} Evolution of the primal iterates $x$ within the optimization landscape for our implicit approach (red circles) and the optimal solution (magenta). The bottom plot shows the corresponding dual-slack variables $(\lambda, s)$ for the four constraints. \textbf{Columns 2--4, Top row:} Eigenvalue spectrum, condition number, and iteration-to-iteration KKT matrix changes. Here, we compare the standard explicit matrix (blue), which diverges and experiences drastic changes near the solution, against our implicit matrix (red), which remains strictly bounded and smoothly varying. \textbf{Columns 2--4, Bottom row:} Evolution of the retraction map's derivative across iterations, illustrating how the auxiliary variables (color-coded by constraint) smoothly settle into their respective active or inactive states.}
    \label{fig:fundamentals}
    \vspace{-2mm}
\end{figure*}

\newpage

\begin{figure*}[t]
    \centering
    \input{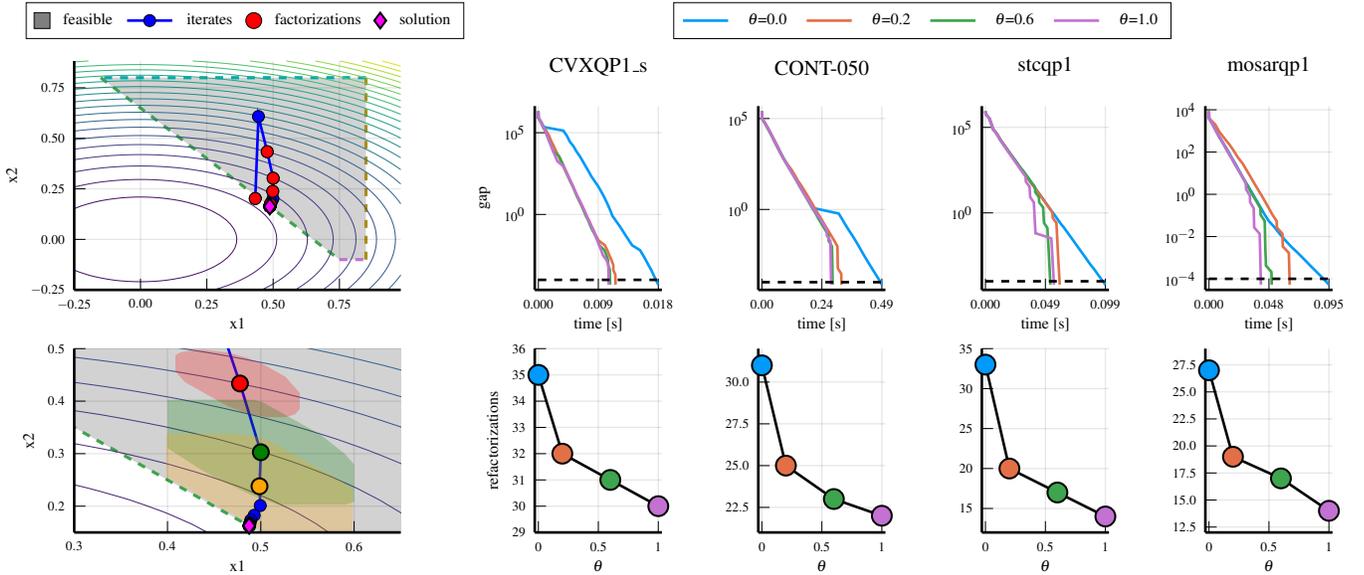}
    \vspace{-8mm}
    \caption{Evaluation of inexact Newton steps, reducing required factorizations by freezing the KKT matrix. \textbf{Column 1:} Synthetic problem iterates (top) and a zoom-in on the final steps (bottom). Red markers highlight explicitly computed factorizations (only 4 of 18 iterations), while the red, green, and yellow regions indicate the validity of the frozen factorizations at those steps. \textbf{Columns 2--4:} Performance on the Maros-Meszaros dataset~\cite{maros1999repository}. \textbf{Top row:} Duality gap reduction over time for different forcing factors $\theta$. \textbf{Bottom row:} Total number of matrix factorizations versus $\theta$. Higher forcing factors enable more aggressive matrix freezing, significantly accelerating overall convergence.}
    \label{fig:inexact_newton}
    \vspace{-5mm}
\end{figure*}
\section{Numerical Experiments}~\label{sec:experiments}
In this section, we present illustrative numerical evaluations to validate the theoretical attributes outlined in the previous section. Throughout these evaluations, we compare the standard primal-dual interior-point baseline~\eqref{eq:explicit_partially_condensed} against our proposed method~\eqref{eq:implicit_partially_condensed}, which uses the softplus retraction map from eq.~\eqref{eq:softplus}. Rather than claiming absolute performance gains—which heavily depend on implementation details and hyperparameter tuning—these experiments are designed to highlight the conceptual benefits of our approach. Our numerical validation is conducted in an orderly and principled manner:

First, we aim to demonstrate the fundamental numerical advantages of our implicit method over the standard explicit formulation. To this end, we use a simple two-dimensional synthetic problem designed to isolate and clearly illustrate the specific numerical phenomena of interest. This strictly convex QP consists solely of four inequality constraints, given by the following data:
\begin{equation} \label{eq:synthetic_qp}
        Q = I_2,\;\; q = 0,\;\;
        A = \begin{bmatrix} 1 & 1 \\ 0 & 1 \\ -1 & 0 \\ 0 & -1 \end{bmatrix}, \;\; b = \begin{bmatrix} 0.65 \\ -0.1 \\ -0.85 \\ -0.8 \end{bmatrix}.
\end{equation}

Second, after validating the numerical benefits of the implicit approach, we focus on algorithmic opportunities that a spectrally bounded linear system offers. In particular, we focus on three specific paradigms: reusing factorizations across iterations, factorization-free iterative solvers and applicability to low-precision arithmetic. These tests are conducted over problems from the well-established Maros-Meszaros dataset~\cite{maros1999repository}. This collection features diverse real-world problems with varying degrees of sparsity, dimensionality, and scaling. 

The implementation is conducted in Julia, and the source code is made publicly available\footnote{Code: \url{https://github.com/jonarriza96/i2pd.jl}}. Every problem instance is equilibrated using Ruiz scaling~\cite{ruiz2001scaling}. 

\subsection{Fundamental benefits of implicit complementarity}
Before demonstrating the practical opportunities enabled by our implicit approach, we first verify its core numerical properties: spectral boundedness and structural stability across iterations. To achieve this, we analyze the synthetic toy problem introduced in~\eqref{eq:synthetic_qp} and compare various numerical metrics between the explicit and implicit formulation. The results are summarized in Fig.~\ref{fig:fundamentals} and discussed in three parts below.

\subsubsection{Evolution of primals, duals, and slacks}
The state iterates corresponding to the implicit solution are shown in the first column of Fig.~\ref{fig:fundamentals}. The top row illustrates the optimization landscape, depicting the contour lines and the gray-shaded feasible region of the inequalities, alongside the evolution of the primal variables $x$ (red) and the final optimal solution (magenta). The second row demonstrates the evolution of the dual-slack pairs ($\lambda, s$) for each constraint. As seen in the primal plot, only the first constraint becomes active ($s=0$ and $\lambda>0$), while the remaining constraints remain inactive ($s>0$ and $\lambda=0$).

\subsubsection{Evolution of retraction}
The evolution of the retraction map's derivative---the only component that changes in the implicit matrix~\eqref{eq:implicit_partially_condensed_softrelu} across iterations---is shown in columns 2, 3, and 4 of the second row. Here, we observe that the derivative of the softplus function becomes sharper as the iterates settle on their respective sides ($1$ for the single active constraint and $0$ for the remaining inactive constraints).

\subsubsection{Conditioning and matrix differences}
The fundamental advantages of our implicit representation are highlighted in columns 2, 3, and 4 of the first row. As expected, the eigenvalue spectrum of the explicit matrix diverges, whereas the spectrum of the implicit matrix remains strictly bounded. This is further evidenced by the condition number, which explodes for the explicit formulation but remains well-conditioned in our approach. 

Additionally, comparing the iteration-to-iteration changes in the KKT matrices provides valuable insight. In the explicit case, the only varying component is $\Lambda^{-1}S$, while in the implicit case, it is $B_\mu(-v)$. During the final iterations near the optimal solution, the explicit matrix undergoes drastic changes, which explains its reliance on frequent, expensive refactorizations. Conversely, the implicit matrix barely changes as convergence is approached. This stability directly correlates with our previous analysis of the retraction map's evolution: the augmented variable associated with each constraint settles on the appropriate side of the retraction map early on, avoiding abrupt jumps and yielding a very smooth transition of the KKT matrix.

To summarize, our numerical evaluation on the synthetic problem validates two key properties of the matrix associated with our implicit representation: (\textit{P1}) it experiences minimal changes as the iterates approach the optimal solution, and (\textit{P2}) it remains spectrally bounded throughout the solving process.


\begin{figure*}[t]
    \centering
    \input{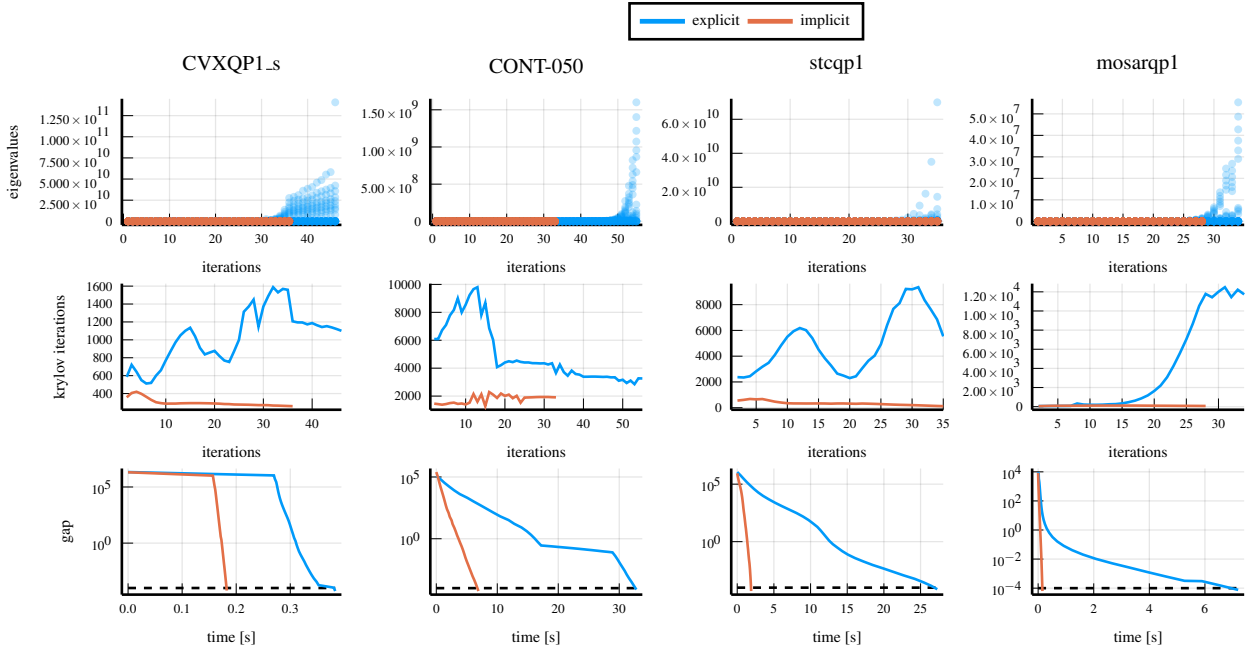}
    \vspace{-6mm}
    \caption{Evaluation of the MINRES iterative solver on the Maros-Meszaros dataset~\cite{maros1999repository} for the explicit (blue) and implicit (red) formulations. \textbf{Top row:} Eigenvalue spectrum across interior-point iterations. \textbf{Middle row:} Number of inner Krylov (MINRES) iterations required per interior-point step. \textbf{Bottom row:} Duality gap reduction versus total solve time. The implicit method's spectral boundedness prevents the iteration explosion and severe slowdowns observed in the explicit method.}
    \vspace{-5mm}
    \label{fig:iterative_methods}
\end{figure*}


\subsection{Reusing factorizations via inexact Newton}

Building on the findings from the previous subsection, we now focus on property (\textit{P1}): the minimal variation of the KKT matrix. By exploiting this gradual evolution, we can reuse factorizations across multiple iterations---effectively employing an inexact Newton method---which ultimately yields significant computational speedups. 

To formalize this inexact Newton strategy, we first compactly denote the partially condensed system in~\eqref{eq:implicit_partially_condensed} as $J(v) \Delta \tilde{z} = -\tilde{r}$, where $\Delta \tilde{z}$ and $\tilde{r}$ are the corresponding condensed state and residual vectors. Let $J^*$ represent the implicit matrix that remains frozen from a previous iterate. We compute an approximate step $\Delta \tilde{z}_k$ by solving the frozen system $J^* \Delta \tilde{z}_k = -\tilde{r}_k$.

To ensure sufficient progress, this approximate step must satisfy the standard inexact Newton condition $\| \tilde{r}_k + J_k \Delta \tilde{z}_k \| \leq \theta_k \| \tilde{r}_k \|$, where $\theta_k$ is a forcing term that dictates the permissible inexactness (see~\cite[Chapter 11]{nocedal2006numerical}). By substituting $\tilde{r}_k = -J^* \Delta \tilde{z}_k$ into the left-hand side and applying the structure of our matrix from~\eqref{eq:implicit_partially_condensed_softrelu}, this requirement reduces to:
\begin{equation}\label{eq:inexact_newton}
    \|\left(B_{\mu}(-v^*)-B_{\mu}(-v_k)\right)\Delta v_k\|\ \leq \theta_k \|\tilde{r}_{2,k}\|\,,
\end{equation}
where $\tilde{r}_{2,k}$ is the second component of $\tilde{r}_{k}$. We evaluate this condition at each step to determine whether to compute a new factorization or reuse the frozen one. To demonstrate the computational benefits of this approach, we divide our analysis into two parts.

First, we revisit the synthetic problem defined in~\eqref{eq:synthetic_qp}. Its simplicity provides a clear illustration of how and when factorizations can be reused. The corresponding results are shown in the first column of Fig.~\ref{fig:inexact_newton}. Specifically, before solving the linear system at each iteration, we evaluate the inexact Newton condition in~\eqref{eq:inexact_newton}. If the condition is satisfied, we bypass the expensive factorization step and reuse the frozen matrix. The top-row plot displays the optimization iterates in blue, highlighting in red the specific iterations where a new factorization was explicitly computed---notably, this expensive step was only required in 4 out of the 18 total iterations. The bottom-row plot zooms in on the final iterates, illustrating the regions where the frozen factorizations remain valid for the last three updates.

Second, we extend our analysis to problems from the Maros-Meszaros dataset to demonstrate how the benefits of reusing factorizations generalize to practical, real-world scenarios. These results are presented in the second through fourth columns of Fig.~\ref{fig:inexact_newton}. The first row illustrates the duality gap reduction with respect to total solve time, while the second row plots the total number of factorizations against various constant forcing terms $\theta$. While developing a sophisticated, adaptive strategy for the forcing term is beyond the scope of this work (see~\cite[Chapter 11]{nocedal2006numerical}), these preliminary results clearly indicate that our implicit formulation successfully enables factorization reuse. This capability contributes to substantial computational speedups, a distinct advantage that is unattainable with the standard explicit representation.

\subsection{Factorization-free iterative methods}\label{subsec:exp_factor_free}

We now shift our focus to property (\textit{P2}): the spectral boundedness of the implicit matrix. This characteristic is a critical enabler for factorization-free iterative solvers, which are notoriously sensitive to ill-conditioned systems and typically require expensive preconditioners~\cite{gondzio2012matrix}. Because our implicit formulation is spectrally bounded by construction, it naturally overcomes this limitation. 

To validate this claim, we evaluate the MINRES algorithm on the explicit system in~\eqref{eq:explicit_partially_condensed} and our implicit system in~\eqref{eq:implicit_partially_condensed} across various problems from the Maros-Meszaros dataset. For a fair comparison, we precondition both implementations with a diagonal block-Jacobi method~\cite{anzt2019adaptive}. Though we only report results for the block-Jacobi method, more sophisticated Schur complement-based preconditioners yielded comparable outcomes. Both the relative and absolute tolerances for the iterative solver were set to $10^{-10}$.%

The results are presented in Fig.~\ref{fig:iterative_methods}. The top row illustrates the eigenvalue spectrum, the middle row displays the number of Krylov iterations, and the bottom row plots the evolution of the duality gap with respect to total solve time. As expected, the eigenvalue spectrum of the explicit representation diverges as the interior-point iterations progress, whereas the spectrum of the implicit representation remains strictly bounded. This aligns with our theoretical expectations and generalizes the observations from Fig.~\ref{fig:fundamentals} to practical, real-world problems.

The practical implications of this spectral boundedness are evident in the middle and bottom rows. The number of Krylov iterations (i.e., the internal iterations within MINRES) is consistently higher for the explicit formulation than for the implicit one. In fact, in all cases except for CONT-050, the number of required MINRES iterations grows as the explicit matrix becomes increasingly ill-conditioned. Conversely, for the implicit case, the number of Krylov iterations remains remarkably flat and even exhibits a slight reduction as convergence is approached. Because the explicit method requires an increasing number of internal Krylov iterations per outer step, its total solve times are considerably longer than those of the implicit implementation, as shown in the bottom row.

\subsection{Machine tolerance at any floating-point precision}

To conclude our numerical validation, we highlight an additional benefit of property \textit{(P2)}: a spectrally bounded KKT matrix enables high-accuracy solutions even on lower-precision hardware. We demonstrate this by comparing the explicit and implicit formulations when solving the linear system---the most computationally demanding step of the algorithm---using both double (float64) and single (float32) precision. In these tests, we bypass standard convergence criteria and allow the algorithms to run until they stall due to numerical limits. Note that the line search is consistently evaluated in double precision for both methods.

The results for two representative problems from the Maros-Meszaros dataset are shown in Fig.~\ref{fig:precision}. Under float64, both the explicit and implicit methods reach a machine precision of approximately $10^{-15}$. However, when restricted to float32, the explicit method stagnates near $10^{-5}$, whereas our implicit formulation successfully achieves the same $10^{-15}$ tolerance. This stark contrast illustrates how a spectrally bounded linear system can decouple a method's ultimate accuracy from the precision of the linear solver. Ultimately, this demonstrates the significant potential of the implicit representation to deploy high-precision interior-point methods on hardware optimized for single-precision arithmetic.

 \begin{figure}[t]
	\centering
    \input{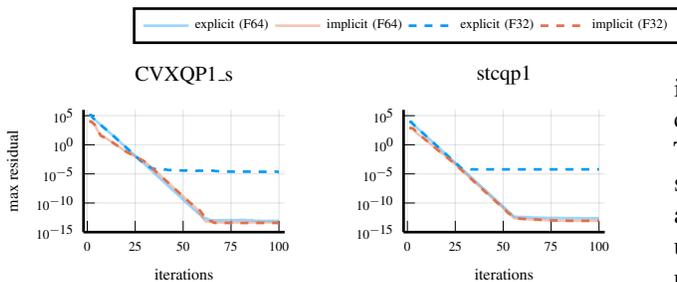}
    \vspace{-8mm}
    \caption{Evolution of the maximum KKT residuals (see eq.~\eqref{eq:residuals}) for the explicit (blue) and implicit (red) formulations under float64 (solid) and float32 (dashed) solver precisions. The spectrally bounded implicit system successfully achieves high accuracy despite low-precision arithmetic.}\label{fig:precision}
    \vspace{-5mm}
\end{figure}

\section{Discussion and Future Work}\label{sec:future}
The implicit method presented in this work offers a scalable framework for modern quadratic programming. By removing the eigenvalue divergence and ill-conditioning that affect standard interior-point methods, it naturally enables inexact Newton steps with reusable factorizations, factorization-free iterative solvers, and low-precision arithmetic, while retaining the rapid convergence and high accuracy that make interior-point methods attractive.

The opportunities highlighted above should be viewed as illustrative rather than exhaustive. In our view, they point to a broader set of possibilities that remains largely unexplored. On the theoretical side, important next steps include incorporating predictor-corrector strategies to reduce the number of interior-point iterations and extending the framework to other cones, such as the positive semi-definite (PSD) cone, where the shortcomings of standard interior-point methods are often even more pronounced. On the practical side, realizing the full potential of this approach will require a carefully engineered software implementation capable of exploiting modern accelerated and highly parallel hardware.

Taken together, these extensions and the distinctive properties of the implicit formulation—factorization-free linear solves, reuse of factorizations across iterations, and robustness to low-precision arithmetic—suggest that this line of work holds strong promise for bringing interior-point methods into the modern era of massive-scale optimization.
\section{Conclusions}\label{sec:conclusion}
We have presented a novel approach for primal-dual interior-point methods in quadratic programming, where complementarity is guaranteed by construction, i.e., implicitly. The underlying linear system is spectrally bounded and structurally well-behaved, and thereby, it evolves smoothly across iterations. This opens new opportunities for solving the underlying linear system, whether by reusing factorizations, utilizing factorization-free methods, or achieving high accuracy even with low-precision arithmetic. For these reasons, we believe this implicit representation might be the key to bringing the benefits of interior-point methods into large-scale optimization problems.
\printbibliography{}

\end{document}